\documentclass[11pt]{amsart} 
\usepackage{a4wide}
\usepackage{amsmath,amssymb,graphicx}
\usepackage{mathrsfs}
\usepackage{float}
\usepackage{dsfont}
\usepackage{comment}
\usepackage[T1]{fontenc}
\usepackage[utf8]{inputenc} 
\usepackage{enumitem}
\usepackage{subcaption}
\usepackage{hyperref}
\usepackage{tikz}
\usetikzlibrary{decorations.pathreplacing,decorations.markings,positioning,arrows.meta}
\usetikzlibrary{matrix}
\usepackage{stmaryrd}
\usepackage{xcolor}
\newcommand{\ack}{\section*{Acknowledgements}}
\def\di{\displaystyle}
\def\N{\mathbb{N}}
\def\Z{\mathbb{Z}}
\def\R{\mathbb{R}}
\def\Q{\mathbb{Q}}

\def\L{\mathscr{L}}

\def\di{\displaystyle}
\def\T{\mathbb{T}}

\def\ds{\displaystyle}

\newtheorem{theorem}{Theorem}[section]
\newtheorem{lemma}[theorem]{Lemma}

\begin{document}
\title[]{Diffusion equations with spatially dependent coefficients and fractal Cauer-type networks}
\author{Jacky Cresson$^1$, Anna Szafra\'{n}ska$^2$}

\begin{abstract}
We give a self-contained proof of the connection existing between diffusion equations with spatially dependent coefficients and fractal Cauer-type networks initiated by J. Sabatier in 2020 and discussed in more details in [J. Sabatier and al., Fractional behaviours modelling, Springer, 2022].   
\end{abstract}

 \maketitle

$^1$ Laboratoire de mathématiques et leurs applications, UMR CNRS 5142, Université de Pau et des Pays de l'Adour-E2S

$^2$ Institute of Applied Mathematics, Gdańsk University of Technology, G. Narutowicz Street 11/12, 80-233 Gdańsk, Poland
\tableofcontents

\section{Introduction}

Diffusion equations with spatially dependent coefficients were recently used to model power-law behaviors observed in Hydrogen storage experiments by V. Tartaglione in (\cite{tartaglione},Chap.4, 4.2.4). The type of diffusion equation used in this work is very particular and is related under a specific spatial discretization to Cauer type networks (see \cite{tartaglione},Chap.2,2.6). It was first discussed by J. Sabatier in \cite{sab2020}. These models seems to provide a good alternative to more conventional modeling using fractional models, meaning particular form of the transfer function admitting a fractional form (see \cite{tartaglione}, equation (4.10)).\\

The aim of the present work is to provide complete proofs and a self-contained presentation for relations and results presented in \cite{sab2020,tartaglione}. In particular, we are interested in a global theorem which precises the connection existing between Cauer-type networks and in particular fractal Cauer-type networks and diffusion equation with spatially dependent coefficients. Moreover, we also prove the functional relations satisfied by the transfer function of a fractal Cauer-type network allowing us to characterize the power-law type behaviors which can be expected. \\

The paper is organized as follows: In Section 2, we define the class of diffusion equations under study as well as the equations obtained under a spatial discretization. Section 3 deals with Cauer type ladder networks and the connection with discretizations of specific diffusion equations. Section 4 gives an explicit expression for the transfer function of a Cauer type ladder network and section 5 specify all the previous results for Fractal Cauer type ladder networks. Section 6 is devoted to a characterization of the fractional behavior of fractal Cauer type ladder network and Section 7 gives some perspectives of this work.

\section{Diffusion equation with variable coefficient}
 
Let $u:\R_+\times \R_+ \to \R$, $u=u(t,z)$, is a solution of the following {\bf diffusion equation} with spatially dependant coefficients: 
\begin{equation} 
\label{diff}
\partial_t u = \gamma(z) \partial_z \Big(\beta(z)\partial_z u\Big)
\end{equation}
where $\gamma:\R_+\to \R$, $\beta:\R_+\to \R$ and let $u(0,z) = u_0(z)$. 

Taking the Laplace transform of (\ref{diff})
\begin{equation}
U(s,z) = \L(u)(s),
\end{equation}
where $\L$ is the Laplace transform, we obtain
\begin{equation} 
\label{lap}
sU = \gamma(z) \partial_z\Big(\beta(z)\partial_z U\Big).
\end{equation}

In \cite{ac}, following a previous work of A. Oustaloup \cite{Oustaloup}, we have proved that discretization over a geometric space-scale of the inverse Fourier transform of the diffusion equation with constant coefficient can be interpreted as a fractal RL ladder network. \\

In the following, we give the properties satisfied by a space-discretization of the diffusion equation \eqref{lap} over a uniform space-scale.

\subsection{Space discretization of diffusion equation}

Let $h>0$, we denote by $\T_z$ the uniform space-scale defined by the set of points $z_k = kh$, $k\in \N$.\\ 

Let $\varphi(s,z)=\beta(z)\partial_z U$, then we consider the following space-discretizations $\varphi_h, U_h$ of functions $\varphi$ and $U$

\begin{equation}
\begin{array}{rcl}
\vspace{2mm}
\varphi_h (s,kh) &=& \beta(kh) \ds\frac{1}{h}\Big(U_h (s,(k+1)h) - U_h(s,kh)\Big)\\
s U_h(s,kh) &=& \gamma(kh)\ds\frac{1}{h}\Big(\varphi_h(s,kh)-\varphi_h(s,(k-1)h)\Big).
\end{array}
\end{equation}
Therefore, $(\varphi_h,U_h)$ satisfy for all $k\in \N$, $k\geq 1$, the following relations
\begin{equation}\label{dsys}
\begin{array}{rcl}
\vspace{2mm}
\varphi_h(s,kh) &=& \beta(kh)\Delta_+[U_h](s,kh)\\
s U_h(s,kh) &=& \gamma(kh)\Delta_-[\varphi_h](s,kh),
\end{array}
\end{equation}
where $\Delta_+$, $\Delta_-$ are the forward and backward discrete operators defined for all $g\in C (\T_z ,\R^d )$ by 
\begin{equation}
\Delta_+ [f](z_k)=\di\frac{f(z_{k+1})-f(z_k)}{h}\ \ \mbox{\rm and}\ \ 
\Delta_- [f](z_k)=\di\frac{f(z_k )-f(z_{k-1})}{h},
\end{equation}
respectively.\\

The space-discretization of equation \eqref{lap} over the space scale $\T_z$ is then given by 
\begin{equation}
\label{deq}
sU_h(s,kh) = \gamma(kh)\Delta_-\Big[\beta(kh)\Delta_+[U_h]\Big](s,kh),\;\;\; k\geq 1.
\end{equation}

The next section gives formal expression for the recursive relation between voltage and current in a Cauer-type ladder network. We will see that the discrete form of the diffusion equation (\ref{deq}) over the space scale $\T_z$ can be interpreted as a specific Cauer-type ladder network.

\section{Cauer type ladder networks}

\subsection{Classical relations on electronic circuits}

Let us denote by $Z$ and $Y$ the impedance and admittance, respectively. We have $Z=\ds\frac{1}{Y}$. In the following, by $S$ we denote an equivalent to one of any component of electronic circuit: resistance ($R$), inductance ($L$), capacitance ($C$). 

We introduce two kinds of the basic relations of the impedance with the current $I$ and the voltage $U$.
\begin{itemize}
\item[$\bullet$] series type relation:
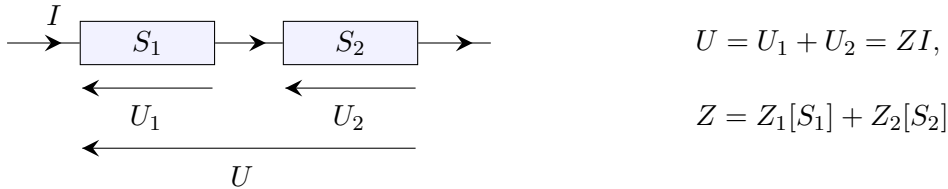
\begin{figure}[!ht]

	\begin{tikzpicture}
		[-,auto,thin,
			boxR/.style={rectangle,text width=4em,fill=blue!5,draw,align=center},
			boxC/.style={rectangle,text width=1em,fill=orange!5,draw,align=center},
			mid arrow/.style={draw, postaction={decorate},
				decoration={
					markings, mark=at position 0.75 with {\arrow[scale=2]{>}}}},
			end arrow/.style={draw, postaction={decorate},
				decoration={
					markings, mark=at position 1 with {\arrow[scale=2]{>}}}}	
			]

		\node[] (empty0) {};
		\node[boxR] (S1) [right of = empty0,node distance = 2cm] {$S_1$};
		\node[boxR] (S2) [right of = S1,node distance = 2.7cm] {$S_2$};
		\node[] (empty1) [right of = S2,node distance = 2cm]{};
		
		\node[] (U0) [below of = empty0,node distance = 0.6cm] {};
		\node[] (U1) [right of = U0,node distance = 1cm] {};
		\node[] (U11) [below of = U1,node distance = 0.8cm] {};
		\node[] (U2) [right of = U1,node distance = 2cm] {};
		\node[] (U3) [right of = U2,node distance = 0.7cm] {};
		\node[] (U4) [right of = U3,node distance = 2cm] {};
		\node[] (U41) [below of = U4,node distance = 0.8cm] {};

		\node[] (pattern) [right of = empty1,node distance = 4.5cm,text width=10em]{$U=U_1+U_2=Z I$,};
		\node[] (pattern1) [below of = pattern,node distance = 1cm,text width=10em]{$Z = Z_1[S_1]+Z_2[S_2]$};

					
		\draw[mid arrow,>=stealth] (empty0) -- (S1) node[above=1mm,text width=7em] {$I$};
		\draw[mid arrow,>=stealth] (S1) -- (S2) {};		
		\draw[mid arrow,>=stealth] (S2) -- (empty1) {};	
		
		\draw[end arrow,>=stealth] (U2) -- (U1) node[below=1mm,text width = -4em]{$U_1$};
		\draw[end arrow,>=stealth] (U4) -- (U3) node[below=1mm,text width = -4em]{$U_2$};
		
		\draw[end arrow,>=stealth] (U41) -- (U11) node[below=1mm,text width = -11em]{$U$};
					
	\end{tikzpicture}

\end{figure}
\item[$\bullet$] parallel type relation:
\begin{figure}[!ht]

 \begin{tikzpicture}
		[-,auto,thin,
			boxR/.style={rectangle,text width=4em,fill=blue!5,draw,align=center},
			boxC/.style={rectangle,text width=1em,fill=orange!5,draw,align=center},
			connection/.style={inner sep=0,outer sep=0},
			mid arrow/.style={draw, postaction={decorate},
				decoration={
					markings, mark=at position 0.75 with {\arrow[scale=2]{>}}}},
			end arrow/.style={draw, postaction={decorate},
				decoration={
					markings, mark=at position 1 with {\arrow[scale=2]{>}}}}	
			]

		\node[connection] (0) [text width = 0.0001em] {};
		\node[connection] (1) [right of = 0,node distance = 1cm,text width = 0.0001em] {};
		\node[connection] (11) [above of = 1, node distance = 0.8cm] {};
		\node[connection] (12) [below of = 1, node distance = 0.8cm]{};
		
		\node[boxR] (S1) [right of = 11,node distance = 2cm] {$S_1$};
		\node[boxR] (S2) [right of = 12,node distance = 2cm] {$S_2$};
		
		\node[connection] (31) [right of = S1, node distance = 2cm]{};
		\node[connection] (32) [right of = S2, node distance = 2cm]{};
		\node[connection] (3) [below of = 31, node distance = 0.8cm]{};
		\node[connection] (4) [right of = 3, node distance = 1cm]{};
		
		\node[] (U1) [below of = 12, node distance = 0.7cm] {};
		\node[] (U2) [below of = 32, node distance = 0.7cm] {};
		
		\node[] (p) [right of = 31,node distance = 6cm] {};
		\node[] (pattern) [below of = p,node distance = 0.5cm,text width=10em]{$U=Z (I_1+I_2) = ZI$,};
		\node[] (pattern1) [below of = pattern,node distance = 1cm,text width=10em]{$\displaystyle\frac{1}{Z} = \frac{1}{Z_1[S_1]}+\frac{1}{Z_2[S_2]}$};

					
		\draw[mid arrow,>=stealth] (0) -- (1) node[above=1mm,text width=3em] {$I$};
		\draw (11) -- (12) {};		
		\draw[mid arrow,>=stealth] (11) -- (S1) node[above=1mm,text width=8em] {$I_1$};
		\draw[mid arrow,>=stealth] (12) -- (S2) node[above=1mm,text width=8em] {$I_2$};
		\draw (S1) -- (31) {};
		\draw (S2) -- (32) {};
		\draw (31) -- (32) {};
		\draw (3) -- (4) {};
			
		\draw[end arrow,>=stealth] (U2) -- (U1) node[below=1mm,text width = -11em]{$U$};
					
	\end{tikzpicture}

\end{figure}
\end{itemize}

\subsection{Cauer type RC networks}

Let us introduce the recursive synthesis of the Cauer type RC network. 

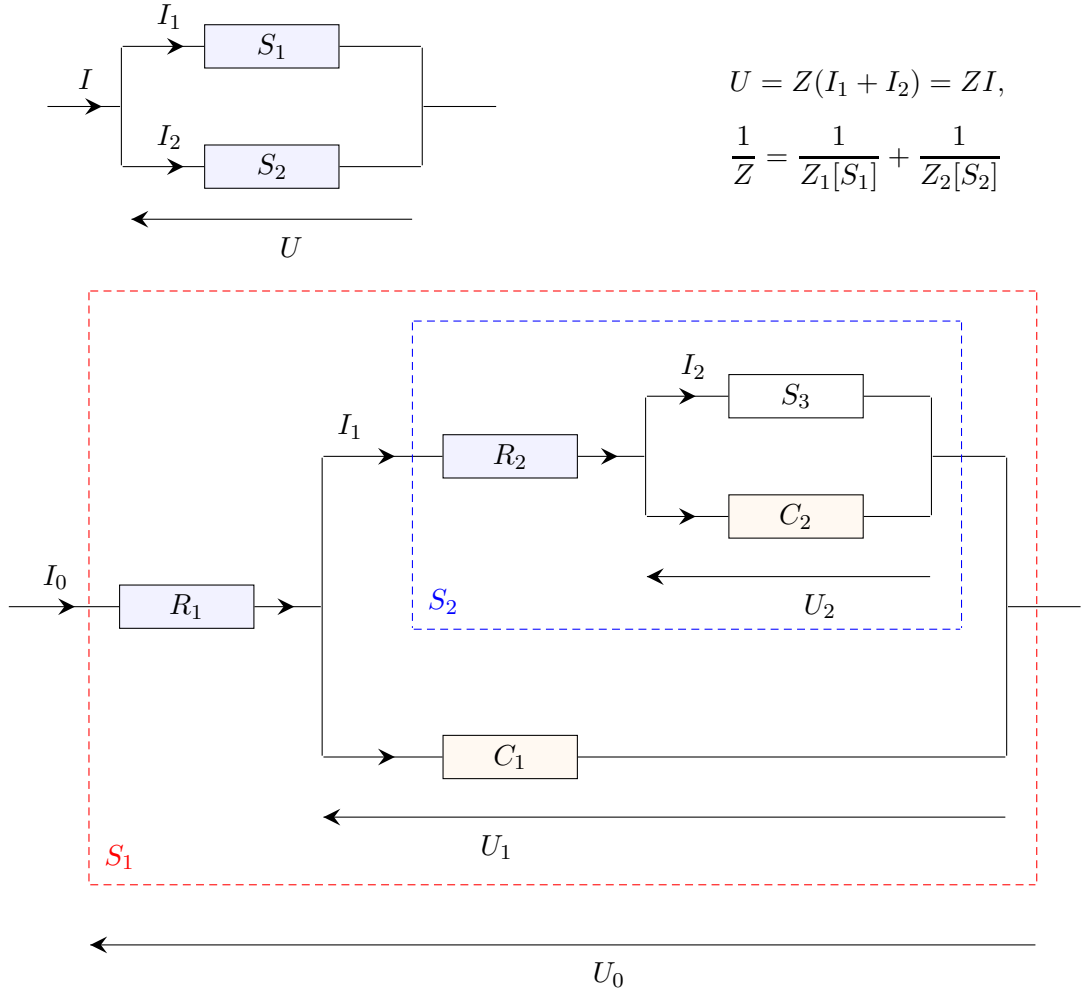
\begin{figure}[!ht]

	\begin{tikzpicture}
		[-,auto,thin,
			boxR/.style={rectangle,text width=4em,fill=blue!5,draw,align=center},
			boxC/.style={rectangle,text width=4em,fill=orange!5,draw,align=center},
			boxS/.style={rectangle,text width=4em,draw,align=center},
			connection/.style={inner sep=0,outer sep=0},
			mid arrow/.style={draw, postaction={decorate},
				decoration={
					markings, mark=at position 0.6 with {\arrow[scale=2]{>}}}},
			end arrow/.style={draw, postaction={decorate},
				decoration={
					markings, mark=at position 1 with {\arrow[scale=2]{>}}}}	
			]

		\node[] (empty0) {};
		\node[boxR] (R1) [right of = empty0,node distance = 2.5cm] {$R_1$};
		\node[connection] (1) [right of = R1,node distance = 1.8cm] {};
		\node[connection] (11) [above of = 1, node distance = 2cm] {};
		\node[connection] (12) [below of = 1, node distance = 2cm]{};
		
		\node[boxR] (R2) [right of = 11, node distance = 2.5cm]{$R_2$};
		\node[boxC] (C1) [right of = 12,node distance = 2.5cm] {$C_1$};
		
		\node[connection] (2) [right of = R2,node distance=1.8cm]{};
		\node[connection] (21) [above of = 2, node distance = 0.8cm] {};
		\node[connection] (22) [below of = 2, node distance = 0.8cm]{};
		
		\node[boxS] (S3) [right of = 21,node distance = 2cm] {$S_{3}$};
		\node[boxC] (C2) [right of = 22,node distance = 2cm] {$C_2$};
		
		\node[connection] (31) [right of = S3, node distance = 1.8cm]{};
		\node[connection] (32) [right of = C2, node distance = 1.8cm]{};
		\node[connection] (3) [below of = 31, node distance = 0.8cm]{};
		\node[connection] (4) [right of = 3, node distance = 1cm]{};
		\node[connection] (5) [below of = 4, node distance = 4cm]{};
		\node[connection] (51) [below of = 4, node distance = 2cm]{};
		\node[connection] (6) [right of = 51, node distance = 1cm]{};
		
		
		\node[connection] (S01) [above of = empty0,node distance = 4.2cm] {};
		\node[connection] (S02) [below of = empty0,node distance = 3.7cm] {};
		\node[connection] (S011) [right of = S01,node distance = 1.2cm] {};
		\node[connection] (S021) [right of = S02,node distance = 1.2cm] {};
		
		\node[connection] (S11) [above of = 6,node distance = 4.2cm] {};
		\node[connection] (S12) [below of = 6,node distance = 3.7cm] {};
		\node[connection] (S111) [left of = S11,node distance = 0.6cm] {};
		\node[connection] (S121) [left of = S12,node distance = 0.6cm] {};
		
		\node[connection] (S21) [above of = 11,node distance = 1.8cm] {};
		\node[connection] (S22) [below of = 11,node distance = 2.3cm] {};
		\node[connection] (S211) [right of = S21,node distance = 1.2cm] {};
		\node[connection] (S221) [right of = S22,node distance = 1.2cm] {};
		
		\node[connection] (S31) [above of = 4,node distance = 1.8cm] {};
		\node[connection] (S32) [below of = 4,node distance = 2.3cm] {};
		\node[connection] (S311) [left of = S31,node distance = 0.6cm] {};
		\node[connection] (S321) [left of = S32,node distance = 0.6cm] {};
		
		\node[connection] (U01) [below of = S021,node distance = 0.8cm] {};
		\node[connection] (U02) [below of = S121,node distance = 0.8cm] {};
		\node[connection] (U11) [below of = 12,node distance = 0.8cm] {};
		\node[connection] (U12) [right of = U11,node distance = 9.1cm] {};
		\node[connection] (U21) [below of = 22,node distance = 0.8cm] {};
		\node[connection] (U22) [right of = U21,node distance = 3.8cm] {};
%
%
%
			
					
		\draw[mid arrow,>=stealth] (empty0) -- (R1) node[above=1mm,text width=10em] {$I_{0}$};
		\draw[mid arrow,>=stealth] (R1) -- (1) {};		
		\draw (11) -- (12) {};		
		\draw[mid arrow,>=stealth] (11) -- (R2) node[above=1mm,text width=12em] {$I_1$};
		\draw[mid arrow,>=stealth] (12) -- (C1) node[above=1mm,text width=8em] {};
		
		\draw[mid arrow,>=stealth] (R2) -- (2) {};		
		\draw (21) -- (22) {};		
		\draw[mid arrow,>=stealth] (21) -- (S3) node[above=1mm,text width=8em] {$I_2$};
		\draw[mid arrow,>=stealth] (22) -- (C2) node[above=1mm,text width=8em] {};
		
		\draw (S3) -- (31) {};
		\draw (C2) -- (32) {};
		\draw (31) -- (32) {};
		\draw (3) -- (4) {};
		\draw (4) -- (5) {};
		\draw (5) -- (C1) {};
		\draw (51) -- (6) {};
		
		\draw[densely dashed,draw=red] (S011) -- (S021) node[above right=1mm,text=red] {$S_1$};
		\draw[densely dashed,draw=red] (S011) -- (S111);
		\draw[densely dashed,draw=red] (S111) -- (S121);
		\draw[densely dashed,draw=red] (S021) -- (S121);
		
		\draw[densely dashed,draw=blue] (S211) -- (S221) node[above right=1mm,text=blue] {$S_2$};
		\draw[densely dashed,draw=blue] (S211) -- (S311);
		\draw[densely dashed,draw=blue] (S311) -- (S321);
		\draw[densely dashed,draw=blue] (S221) -- (S321);

		\draw[end arrow,>=stealth] (U02) -- (U01) node[below=1mm,text width = -35em]{$U_0$};
		
		\draw[end arrow,>=stealth] (U12) -- (U11) node[below=1mm,text width = -11em]{$U_1$};
		
		\draw[end arrow,>=stealth] (U22) -- (U21) node[below=1mm,text width = -11em]{$U_2$};
					
	\end{tikzpicture}

    \caption{Recursive synthesis of the RC circuit with the series-parallel topology} \label{recursive}
\end{figure}
According to Figure \ref{recursive} we can observe that the RC circuit can be described by the recurrent pattern presented on Figure \ref{recurrent}.
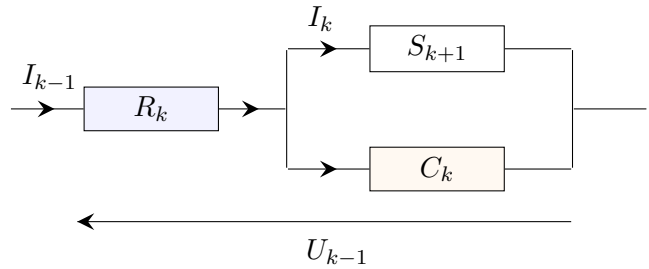
\begin{figure}[!ht]

	\begin{tikzpicture}
		[-,auto,thin,
			boxR/.style={rectangle,text width=4em,fill=blue!5,draw,align=center},
			boxC/.style={rectangle,text width=4em,fill=orange!5,draw,align=center},
			boxS/.style={rectangle,text width=4em,draw,align=center},
			connection/.style={inner sep=0,outer sep=0},
			mid arrow/.style={draw, postaction={decorate},
				decoration={
					markings, mark=at position 0.6 with {\arrow[scale=2]{>}}}},
			end arrow/.style={draw, postaction={decorate},
				decoration={
					markings, mark=at position 1 with {\arrow[scale=2]{>}}}}	
			]

		\node[] (empty0) {};
		\node[connection] (empty1) [right of = empty0, node distance = 1cm] {};
	
		\node[boxR] (Rn) [right of = empty0, node distance = 2cm]{$R_k$};
		
		\node[connection] (1) [right of = Rn,node distance=1.8cm]{};
		\node[connection] (11) [above of = 1, node distance = 0.8cm] {};
		\node[connection] (12) [below of = 1, node distance = 0.8cm]{};
		
		\node[boxS] (Sn1) [right of = 11,node distance = 2cm] {$S_{k+1}$};
		\node[boxC] (Cn) [right of = 12,node distance = 2cm] {$C_k$};
		
		\node[connection] (21) [right of = Sn1, node distance = 1.8cm]{};
		\node[connection] (22) [right of = Cn, node distance = 1.8cm]{};
		\node[connection] (3) [below of = 21, node distance = 0.8cm]{};
		\node[connection] (4) [right of = 3, node distance = 1cm]{};

		\node[connection] (Un1) [below of = empty1,node distance = 1.5cm] {};
		\node[connection] (Un2) [below of = 3,node distance = 1.5cm] {};
%
%
%
			
					
		\draw[mid arrow,>=stealth] (empty0) -- (Rn) node[above=1mm,text width=9em] {$I_{k-1}$};
		\draw[mid arrow,>=stealth] (Rn) -- (1) {};		
		\draw (11) -- (12) {};		
		\draw[mid arrow,>=stealth] (11) -- (Sn1) node[above=1mm,text width=9em] {$I_k$};
		\draw[mid arrow,>=stealth] (12) -- (Cn) node[above=1mm,text width=8em] {};
		
		\draw (Sn1) -- (21) {};
		\draw (Cn) -- (22) {};
		\draw (21) -- (22) {};
		\draw (3) -- (4) {};

		\draw[end arrow,>=stealth] (Un2) -- (Un1) node[below=1mm,text width = -16em]{$U_{k-1}$};
					
	\end{tikzpicture}

    \caption{Recurrent pattern of RC circuit} \label{recurrent}
\end{figure}

We use the classical current and voltage relationships in the components of RC electronic circuit with series-parallel topology (Figures \ref{CkR} and \ref{RkR}).  
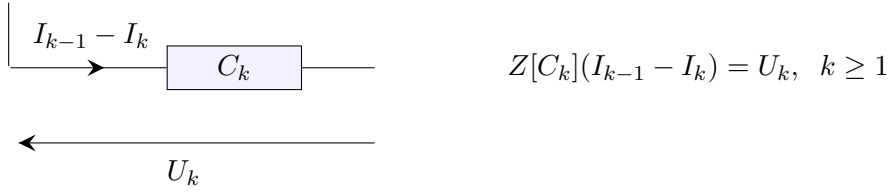
\begin{figure}[!ht]

	\begin{tikzpicture}
	[-,auto,thin,
	boxR/.style={rectangle,text width=4em,fill=blue!5,draw,align=center},
	boxC/.style={rectangle,text width=1em,fill=orange!5,draw,align=center},
	mid arrow/.style={draw, postaction={decorate},
		decoration={
			markings, mark=at position 0.6 with {\arrow[scale=2]{>}}}},
	end arrow/.style={draw, postaction={decorate},
		decoration={
			markings, mark=at position 1 with {\arrow[scale=2]{>}}}}	
	]
	
	\node[inner sep=0pt] (empty0) {};
	\node[] (empty00) [above of = empty0, node distance = 1cm] {};
	\node[boxR] (Ck) [right of = empty0,node distance = 3cm] {$C_k$};
	\node[] (empty1) [right of = Ck,node distance = 2cm]{};
	
	\node[] (U0) [below of = empty0,node distance = 1cm] {};
	\node[] (U1) [right of = U0,node distance = 5cm] {};

	\node[] (pattern) [right of = empty1,node distance = 4.5cm,text width=15em]{$Z[C_k](I_{k-1}-I_k)=U_k,\;\; k\geq 1$};

	
	\draw[mid arrow,>=stealth] (empty0) -- (Ck) node[above=1mm,text width=14em] {$I_{k-1}-I_k$};		
	\draw[] (Ck) -- (empty1) {};	

	\draw[end arrow,>=stealth] (U1) -- (U0) node[below=1mm,text width = -11em]{$U_k$};
	
	\draw[] (empty00) -- (empty0);
	
	\end{tikzpicture}
	
    \caption{Current-voltage relationship in capacitor}
    \label{CkR}
\end{figure}
\begin{figure}[!ht]

	\begin{tikzpicture}
	[-,auto,thin,
	boxR/.style={rectangle,text width=4em,fill=blue!5,draw,align=center},
	boxC/.style={rectangle,text width=1em,fill=orange!5,draw,align=center},
	mid arrow/.style={draw, postaction={decorate},
		decoration={
			markings, mark=at position 0.6 with {\arrow[scale=2]{>}}}},
	end arrow/.style={draw, postaction={decorate},
		decoration={
			markings, mark=at position 1 with {\arrow[scale=2]{>}}}}	
	]
	
	\node[inner sep=0pt] (empty0) {};
	\node[] (empty00) [above of = empty0, node distance = 1cm] {};
	\node[boxR] (Ck) [right of = empty0,node distance = 3cm] {$C_k$};
	\node[] (empty1) [right of = Ck,node distance = 2cm]{};
	
	\node[] (U0) [below of = empty0,node distance = 1cm] {};
	\node[] (U1) [right of = U0,node distance = 5cm] {};

	\node[] (pattern) [right of = empty1,node distance = 4.5cm,text width=15em]{$Z[C_k](I_{k-1}-I_k)=U_k,\;\; k\geq 1$};

	
	\draw[mid arrow,>=stealth] (empty0) -- (Ck) node[above=1mm,text width=14em] {$I_{k-1}-I_k$};		
	\draw[] (Ck) -- (empty1) {};	

	\draw[end arrow,>=stealth] (U1) -- (U0) node[below=1mm,text width = -11em]{$U_k$};
	
	\draw[] (empty00) -- (empty0);
	
	\end{tikzpicture}
	    \caption{Current-voltage relationship in resistance}
    \label{RkR}
\end{figure}
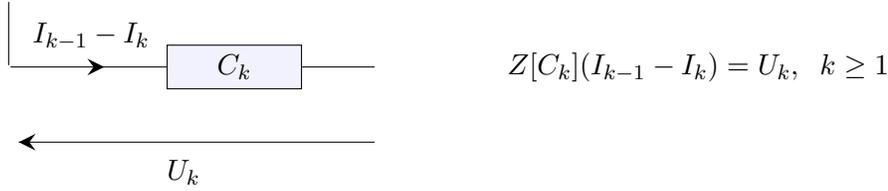

Let $U\in C(\R_+\times\N ;\R), I\in C(\R_+\times\N;\R), R\in C(\N;\R), C\in C(\N;\R)$. We consider a Cauer type RC network (Figure \ref{recurrent}) with 
$$
I_k = I(s,k),\;\; U_k = U(s,k),\;\; R_k = R(k),\;\; C_k = C(k).
$$
for all $k\geq 0$. The impedance of resistor and capacitor in RC circuit is given by
$$
Z[R_k](s) = R_k,\;\;\;\; Z[C_k](s) = \frac{1}{sC_k}
$$
for all $k\geq 0$. Then  the classical relations for elements behaviour of the Cauer type RC network can be written in the form:
\begin{equation} \label{relsys}
\begin{array}{rcl}
\vspace{2mm}
-\ds\frac{1}{R^{\sigma}(k)}\Delta_+[U] (s,k) &=& I(s,k),\;\;\;\; k\geq 0,\\
-\ds\frac{1}{C(k)}\Delta_-[I] (s,k)  &=& s U(s,k), \;\;\;\; k\geq 1.
\end{array}
\end{equation}
where $R^{\sigma}(k) = R(k+1)$.
From (\ref{relsys}) we derive for all $k\geq 1$
\begin{equation} \label{RCeq}
sU(s,k) = -\frac{1}{C(k)}\Delta_-\Big[-\frac{1}{R^{\sigma}(k)}\Delta_+[U](s,k)\Big](s,k).
\end{equation}

This equation is very close to the discrete form of a diffusion equation with spatially dependents coefficients \eqref{deq} discuss in the previous section. We precise this connection in the next Section.

\subsection{Interpretation of a diffusion equation as a Cauer type network}

Comparing (\ref{deq}) and (\ref{RCeq}), we see that (\ref{deq}) can be obtained as a Cauer type network by choosing appropriately the functions $R$ and $C$ in $C(\N ,\R )$. \\

Precisely, let $\T_z=\{z_k=kh\}_{k\geq 0}$. We introduce the mapping $\pi_h :\N \rightarrow \T_z$ defined for all $k\in \N$ by $\pi_h (k)=kh$. We have the following diagram
\begin{center}
\begin{tikzpicture}[>={Classical TikZ Rightarrow[length=2mm,width=0pt 10]}]]
  \matrix (m) [matrix of math nodes,row sep=2em,column sep=5em,minimum width=2em]
    {
    \T_z & \N \\
    \R & \\   
    };
   \path[->]
        (m-1-1) edge node[left] {$f\in C(\T;\R)$} (m-2-1);
    \path[->]
        (m-1-2) edge node [above]{$\pi_h$} (m-1-1);
    \path[->]
        (m-1-2) edge node [below right] {$f\circ \pi_h\in C(\N;\R)$} (m-2-1);
\end{tikzpicture}
\end{center}
allowing to associated to any function $f\in C(\T_z , \R)$ a function in $C(\N ,\R )$.\\

We then choose $R$ and $C$ such that the following relations are satisfied
\begin{equation} \label{R}
-\frac{1}{C} = \beta \circ \pi_h,\;\;\;\; \beta\in C(\T,\R),
\end{equation}
\begin{equation} \label{C}
-\frac{1}{R^{\sigma}} = \gamma \circ \pi_h,\;\;\;\; \gamma\in C(\T,\R).
\end{equation}

We want to select specific Cauer-type ladder networks which are able to exhibits fractional behaviors. In order to do that, we first derive an explicit form of the transfer function associated to a Cauer type ladder networks. 

\section{The transfer function of a Cauer tpe ladder network}

\subsection{Recursive expansion of the transfer function}
Let us denote by $H:C(\R_+\times \N;\R)$ the transfer function
$$
H(s,k) = \frac{I(s,k)}{U(s,k)},\;\;\;\; k\geq 0.
$$
Using (\ref{relsys}) we obtain for all $k\geq 0$
\begin{equation} \label{tranfun}
H(s,k) = \frac{\ds\frac{1}{R(k+1)}}{1+\ds\frac{\ds\frac{1}{sR(k+1)C(k+1)}}{1+\ds\frac{1}{sC(k+1)}H(s,k+1)}}.
\end{equation}
Denoting by $Z_C(s) = Z[C](s) =  \ds\frac{1}{sC}$, (\ref{tranfun}) can be rewritten as
\begin{equation} \label{tranfun1}
H(s,k) = \frac{\ds\frac{1}{R(k+1)}}{1+\ds\frac{\ds Z_{C(k+1)}(sR(k+1))}{1+\ds Z_{C(k+1)}(s)H(s,k+1)}}.
\end{equation}

\subsection{Continued fraction}

In order to write clearly the recursive formula, we introduce the following form of the continued fractions:
\begin{equation} \label{cf}
[a_1,\ldots,a_k] = \frac{a_1}{1+\ds\frac{a_2}{1+\ds\frac{a_3}{1+\cdots\ds\frac{\phantom{x}}{1+\ds\frac{a_{k-1}}{1+a_k}}}}}
\end{equation}

The continued fraction (\ref{cf}) can be rewritten by construction as:
\begin{equation} 
\label{cf1}
[a_1,a_2,\ldots,a_k] = [a_1,[a_2,\ldots,a_k]]
\end{equation}
and we have 
\begin{equation} \label{cf2}
\alpha[a_1,a_2,\ldots,a_k] = [\alpha a_1,a_2,\ldots,a_k].
\end{equation}

\subsection{Continued fraction and the transfer function}

Based on relations (\ref{cf1}) and (\ref{cf2}) we are able to introduce the following lemma:
\begin{lemma} \label{lem1}
The closed formula for transfer function $H\in C(\R_+\times \N)$ has the form
\begin{equation} \label{H0cf}
H(s,0) = \Big[\frac{1}{R(1)},Z_{C(1)}(R(1)s),Z_{C(1)}(R(2)s),\ldots,Z_{C(k)}(R(k)s),Z_{C(k)}(R(k+1)s),\ldots\Big]
\end{equation}
\end{lemma}
\begin{proof}
Using (\ref{tranfun1}) and (\ref{cf}) we can write
$$
H(s,0) = \Big[\ds\frac{1}{R(1)},Z_{C(1)}(R(1)s),Z_{C(1)}(s)H(s,1)\Big].
$$
and for $k\geq 1$ we have
\begin{equation} \label{Hk}
H(s,k) = \Big[\ds\frac{1}{R(k+1)},Z_{C(k+1)}(R(k+1)s),Z_{C(k+1)}(s)H(s,k+1)\Big].
\end{equation}
Replacing $H(s,1)$ in $H(s,0)$ by expression (\ref{Hk}) for $k=1$, we have:
$$
H(s,0) = \Bigg[\ds\frac{1}{R(1)},Z_{C(1)}(R(1)s),Z_{C(1)}(s)\Big[\ds\frac{1}{R(2)},Z_{C(2)}(R(2)s),Z_{C(2)}(s)H(s,2)\Big]\Bigg].
$$
Form (\ref{cf1}) and from $\frac{1}{\alpha}Z_{C(1)}(s) = Z_{C(1)}(\alpha s)$ we obtain
$$
H(s,0)=\Big[\ds\frac{1}{R(1)},Z_{C(1)}(R(1)s),Z_{C(1)}(R(2)s),Z_{C(2)}(R(2)s),Z_{C(2)}(s)H(s,2)\Big].
$$
By recurrence we obtain the form (\ref{H0cf}).
\end{proof}

\section{Fractal Cauer type ladder networks}

Fractal Cauer type ladder networks are specific type of Cauer ladder networks which can be defined using few parameters and exhibiting a rich dynamics. We give an explicit form for the transfer function of a Fractal Cauer type ladder network.

\subsection{Fractal Cauer ladder networks}

We consider a Cauer type ladder network with coefficients $R_k$ and $L_k$ for $k\geq 1$. As usual, we identify $R_k$ and $L_k$ as the values of two functions $R:\N \rightarrow \R$ and $L:\N \rightarrow \R$ at point $k\in \N$.\\

A Cauer type ladder network is said fractal if there exists two constants $\sigma$ and $\rho$, such that we have the scaling relation
\begin{equation} \label{RC}
R(k+1) = \sigma R(k),\;\;\;\; C(k+1) = \rho C(k).
\end{equation}

\subsection{Fractal Cauer transfer function}

Then we can simplify the expression of the transfer function $H(s,0)$ given by (\ref{H0cf}), as follows
\begin{equation} 
\label{H0rs}
H(s,0) = \Big[\frac{1}{R_1},Z(s),\frac{1}{\sigma}Z(s),\frac{1}{\rho\sigma}Z(s),\ldots,\frac{1}{\rho^n \sigma^n}Z(s),\frac{1}{\rho^n \sigma^{n+1}}Z(s),\ldots\Big]
\end{equation}
with $Z(s) = Z_{C(1)}(R(1) s)$.

\subsection{Diffusion equation associated to a fractal Cauer type network}

Using (\ref{R}) and (\ref{C}), we want to explicit the functions $\beta$ and $\gamma$ such that $R$ and $C$ satisfy the \textbf{fractality} condition (\ref{RC}).

By definition, $R$ and $C$ depend on $h$. As a consequence, the \textbf{fractal} constant $\sigma$ and $\rho$ are functions of $h$. Instead of searching $\beta$ and $\gamma$, we look for two \textbf{continuous} functions $r(\zeta)$ and $c(\zeta)$ satisfying
$$
r(kh) = R(k),\;\;\;\; c(kh) = C(k).
$$
Then the scaling relations (\ref{RC}) read as follows for all $h> 0$
\begin{equation} \label{rc}
\begin{array}{rl}
\vspace{1mm}
r((k+1)h) = & \sigma(h) \; r(kh),\\
c((k+1)h) = & \rho(h) \; c(kh),
\end{array} \;\;\; \forall_{k\geq 0}.
\end{equation}
These two conditions are of the form: for a given $\varsigma:\R\to\R$ find $f:\R\to \R$ such that for all $h>0$ we have
\begin{equation} \label{fvar}
f((k+1)h) = \varsigma(h)f(kh),\;\;\;\; \forall_{k\geq 0}.
\end{equation}
From (\ref{fvar}) for $k=0$ we have
$$
f(h) = \varsigma(h) f(0),
$$
and we obtain a definition of function $\varsigma$ in (\ref{fvar}):
\begin{equation} \label{var} 
\varsigma(h) = \frac{f(h)}{f(0)}
\end{equation}
Posing $g(\zeta) = \displaystyle\frac{f(\zeta)}{f(0)}$, the equation (\ref{fvar}) with (\ref{var}) reads as follows
\begin{equation} \label{gkh}
\forall_{h>0}\;\;\;\; g((k+1)h) = g(h)g(kh),\;\;\;\; \forall_{k\geq 0}.
\end{equation}
This functional relation has very particular solution:
\begin{theorem}
A continuous real valued function $g$ is a solution of (\ref{gkh}) if anf only if it there exists a constant $\lambda \in \R$ such that
\begin{equation} \label{gexp}
g(\zeta) = e^{\zeta\lambda}.
\end{equation}
\end{theorem}
\begin{proof}
First we proof that the functional relation (\ref{gkh}) is satisfied.
This proof consists with three steps:
\begin{itemize}
\item[1)] $\displaystyle\forall_{k,k'\in \N}$
$$
g((k+k')\zeta) = g(k\zeta)g(k'\zeta);
$$
\item[2)] $\displaystyle\forall_{a,b\in \Q}$
$$
g(a+b) = g(a)g(b);
$$
\item[3)] $\displaystyle\forall_{u,v\in \R}$
$$
g(u+v) = g(k\zeta)g(k'\zeta);
$$

\end{itemize}

The first equality is proved by induction. Indeed, we have
\begin{equation}
g((k+k')z)=g((k+k'-1)z+z)=g((k+k'-1)z)g(z) .
\end{equation}
Replacing $g((k+k'-1)z)$ by $g((k+k'-2)z)g(z)$ we obtain
\begin{equation}
g((k+k')z)=g((k+k'-2)z)g(z)g(z) .
\end{equation}
But $g(z)g(z)=g(2z)$ so
\begin{equation}
g((k+k')z)=g((k+k'-2)z)g(2z) .
\end{equation}
By iteration, we obtain for all $i=0,\dots ,k'$, that 
\begin{equation}
g((k+k'-i)z+iz)=g((k+k'-i)z)g(iz).
\end{equation}
In particular, for $i=k'$, we obtain 
\begin{equation}
g((k+k')z)=g(kz)g(k'z) .
\end{equation}
As $z$ can be negative, the result extend to $k,k'\in \Z$. 

The second equality is a simple computation using equality 1). We consider two rational numbers of the form $m/n$ and $p/q$ with $n,q\in \N^*$ and $m,n\in \N$. We have  
\begin{equation}
g \left  ( 
\di\frac{m}{n}+\di\frac{p}{q} 
\right ) 
=g\left ( 
\di\frac{mq+np}{nq} \right ) = 
g\left ( 
(mq+np)\di\frac{1}{nq} \right ) 
\end{equation}
Using equality 1), we then obtain
\begin{equation}
g \left  ( 
\di\frac{m}{n}+\di\frac{p}{q} 
\right ) 
=g\left ( 
mq\di\frac{1}{nq}
\right ) 
g\left ( 
np\di\frac{1}{nq} \right ) =
g\left ( 
\di\frac{m}{n}
\right ) 
g\left ( 
\di\frac{p}{q} \right ) .
\end{equation}

The last equality is proved by density using the continuity of $g$. Let $u,v\in \R^+$ and $u_n$, $v_n$, $n\in \N$ be two sequences of rational numbers converging to $u$ and $v$ respectively. Then, for all $n\in \N$, we have by equality 2) that
\begin{equation}
g(u_n + v_n)=g(u_n ) g(v_n ) .
\end{equation}
Taking the limit when $n$ goes to $+\infty$ and by continuity of $g$, we obtain for all $u,v\in \R$ that 
\begin{equation}
g(u+v)=g(u)g(v) .
\end{equation}
As $g$ is a positive function, the previous equality means that $g$ is a continuous morphism of the additive group $(\R ,+)$ to the multiplicative group $(\R^+ ,\times )$. It is known in this case that $g$ must be of the form 
$g(z)=e^{\lambda z}$ for a given $\lambda\in \R$.
\end{proof}

Using the previous Theorem, we obtain the following structure result:

\begin{theorem}
Continuous positive real valued functions $R$ and $C$ satisfy equality \eqref{R} and \eqref{C}, respectively, if and only if there exists two real constants $\lambda_R$ and $\lambda_C$ such that $R(z)=R_0 e^{\lambda_R z}$ and $C(z)=C_0 \di e^{\lambda_C z}$
\end{theorem}

As a consequence, Fractal Cauer type networks are associated to diffusion equations characterized by 
$\beta (z)=-\beta_0 \di e^{-\lambda_R z}$ and $\gamma (z)=-\gamma_0 \di e^{-\lambda_C z}$. \\

We then obtain the following connection between Fractal Cauer type networks and generalized diffusion equations:

\begin{theorem}
Fractal Cauer type networks with characteristics $\lambda_R$ and $\lambda_C$ are in one to one correspondence for all $h>0$ with the discretization in space along the space-scale $\{ kh \}$, $k\in \N$ of the diffusion equation with non-constant diffusion coefficient
\begin{equation}
\partial_t u =-\gamma_0 \di e^{-\lambda C z} 
\di\partial_z \left ( 
-\beta_0 \di e^{-\lambda_R z} \partial_z u 
\right ) .
\end{equation}
\end{theorem}

\section{Fractional behavior of fractal Cauer networks}

A fractional behavior is associated to a transfer function of the form
\begin{equation}
H(s)=cs^{\nu},
\end{equation}
where $c$ and $\nu$ are two constants. The main problem is to determine if a fractal Cauer type systems exhibits a fractional behavior and if yes, to characterize the exponent. \\

We first derive structural relations satisfied by the transfer function of fractal Causer networks and we then determine the possible fractional behaviors.

\subsection{Structural relations}

Let us denote by $g(z,\sigma ,\rho )$ the function
\begin{equation}
\label{eq19}
g(z,\sigma ,\rho )=[z, z\sigma^{-1}, z (\sigma \rho )^{-1} , z (\sigma^2 \rho )^{-1}, z(\sigma^2 \rho^2 )^{-1},\dots ] .
\end{equation}

Then, using equation \eqref{H0rs}, we have
\begin{equation}
H(s,\sigma ,\rho ) =[R_1^{-1}, g(Z(s),\sigma ,\rho) ] .
\end{equation}

We have the following functional relation satisfied by the function $g$:

\begin{theorem}
\label{funcrelg}
The function $g$ satisfies 
\begin{equation}
\label{eq21}
g(Z(s),\sigma ,\rho ) =\di\frac{Z(z)}{1+g(Z(\sigma s ),\rho ,\sigma )} .
\end{equation}
\end{theorem}

It is important to notice the permutation between $\sigma$ and $\rho$ in the formula \eqref{eq21}. Formula \eqref{funcrelg} was first stated by J. Sabatier in (\cite{sab2020}, Property 1 p.249) without a formal proof.

\begin{proof}
Using equality \eqref{cf1} we can write
\begin{equation}
g(Z(s),\sigma ,\rho )= [Z(s), [Z(s)\sigma^{-1}, Z(s) (\sigma \rho )^{-1}, \dots , Z(s) (\sigma \rho )^{-n} , Z(s) \sigma^{-1} (\sigma \rho )^{-n},\dots ]] . 
\end{equation}
As we have $Z(\sigma s )=Z(s)\sigma^{-1}$, we obtain
\begin{equation}
\left .
\begin{array}{lll}
[Z(s)\sigma^{-1}, \dots , Z(s) (\sigma \rho )^{-n} , Z(s) \sigma^{-1} (\sigma \rho )^{-n},\dots ] & = &  
[Z(\sigma s ) , Z(\sigma s) \rho^{-1}, \dots , Z(\sigma s ) (\rho \sigma )^{-n}, Z(\sigma s) \rho^{-1} (\rho \sigma )^{-n},\dots ],\\
 & = & g(Z(\sigma s),\rho ,\sigma ) .
\end{array}
\right .
\end{equation}
This concludes the proof.
\end{proof}

The previous functional relation is difficult to study directly. We can look for a simplified version of it following an idea already used by A. Oustaloup in [\cite{Oustaloup},p.208]:\\

Let us assume that $s$ is such that $g$ satisfies 
\begin{equation}
\mid g(Z(s),\rho ,\sigma ) \mid >>1 
\end{equation}
or equivalently that 
\begin{equation}
\mid g(Z(s),\rho ,\sigma ) \mid^{-1} << 1. 
\end{equation}
Let $1>>\epsilon >0$ be such that 
\begin{equation}
\mid g(Z(s),\rho ,\sigma ) \mid^{-1} < \epsilon. 
\end{equation}
We consider the simplified functional relation 
\begin{equation}
\label{funceps}
g(Z(s),\sigma ,\rho ) g(Z(\sigma s ),\rho ,\sigma ) =\di\frac{Z(s)}{1+\epsilon},
\end{equation}
corresponding to an approximation of \eqref{eq21}. \\

In the next Section, we characterize the fractional behaviors compatible with the functional relation \eqref{funceps}.

\subsection{Fractional behaviors}

We look for solutions of \eqref{funceps} of the form
\begin{equation}
\label{fracb}
g(Z(s),\sigma ,\rho ;\epsilon )= K(\sigma ,\rho ;\epsilon ) Z(s)^{n(\sigma ,\rho )} .    
\end{equation}

Our goal is to determine the possible values for the exponent $n(\sigma ,\rho )$. \\

Replacing $g$ by its expression \eqref{fracb} in \eqref{funceps}, we obtain:

\begin{equation}
K(\sigma ,\rho ;\epsilon ) K(\rho ,\sigma ;\epsilon ) Z(s)^{n(\sigma ,\rho)+n(\rho ,\sigma )} \sigma^{-n(\rho ,\sigma )} =\di\frac{Z(s)}{1+\epsilon} .
\end{equation}

As a consequence, we obtain by identification the following system:

\begin{equation}
\label{sys}
\left .
\begin{array}{lll}
K(\sigma ,\rho ;\epsilon ) K(\rho ,\sigma ;\epsilon ) & = & \di\frac{\sigma^{n(\rho ,\sigma )}}{1+\epsilon} ,\\
n(\sigma ,\rho +n(\rho ,\sigma ) & = & 1.
\end{array}
\right .
\end{equation}

The first equation of \eqref{sys} must be valid by exchanging the role of $\sigma$ and $\rho$ due to the symmetry of the left hand side. As a consequence, we must have:

\begin{equation}
\label{intermed}
\di\frac{\sigma^{n(\rho ,\sigma )}}{1+\epsilon}
=\di\frac{\rho^{n(\sigma ,\rho )}}{1+\epsilon} .
\end{equation}

As $n(\rho ,\sigma )=1-n(\sigma ,\rho )$ by the second equation of \eqref{sys}, we deduce by replacing in \eqref{intermed} that

\begin{equation}
 \sigma =(\sigma \rho )^{n(\sigma ,\rho )}  . 
\end{equation}

We then deduce the following theorem:

\begin{theorem}
Fractional behaviors of the form \eqref{fracb} satisfying the simplified functional relation \eqref{funceps} possess a fractional exponent given by 
\begin{equation}
n(\sigma ,\rho ) =\di\frac{\ln (\sigma )}{\ln (\sigma \rho )} .
\end{equation}
\end{theorem}

The previous result appear in (\cite{sab2020},Appendix 2, (A2.28)) as a consequences of relations satisfied by a formal expansion of the function $g$ (see \cite{sab2020},Theorem 1).  

\section{Conclusion and perspectives}

The previous results can be used in a variety of situations as already stressed by J. Sabatier in \cite{sab2020} and V. Tartaglione et al. in \cite{tartaglione,Sabatier} covering application to the modeling of diffusion in materials used for the storage of Hydrogen. The functional relations satisfied by the transfer function of a Fractal Cauer type ladder networks is very specific and the previous discussion is only a first step in its study. In particular, two mathematical directions deserve more attention:

\begin{itemize}
\item J. Sabatier in (\cite{sab2020},Theorem 1) and V. Tartagione et al. in \cite{Sabatier} discuss a result concerning the asymptotic behavior of the transfer function using an explicit expression (see \cite{Sabatier}, Theorem 7.1. and details in Appendix B). The proof is based on an assumed form for the transfer function which is not proved in the paper. A complete and detailed discussion of this result is then needed.

\item J. Sabatier \cite{sab2020} discuss distributions, i.e. relations between the parameters of the Cauer type networks or equivalently different space scale discretization of a diffusion equation with spatially dependent coefficient, which are going beyond geometric distribution (see \cite{sab2020},p.250) and leading to power law behaviour. It will be interesting to give a full characterization of the possible distributions as well as the form of the associated coefficients corresponding to the class discuss in \cite{sab2020}.
\end{itemize}

From the application point of view, it would be nice to give full characterization of the diffusion properties of some specific material like COF (Covalent Organic Framework) or MOF (Metal-Organic Framework) (see \cite{review} for a review) using the previous approach.

\ack A. Szafra\'{n}ska thanks the National Science Center for the financial support in the form of the research project No. 2021/05/X/ST1/00332 and J. Cresson thanks the GDR CNRS no. 2043, Géométrie différentielle et Mécanique and the fédération MARGAUx (FR 2045) for supports and discussion with J. Sabatier, V. Tartaglione and C. Farges at the IMS Bordeaux.

\bibliographystyle{acmurl}

\end{document}